\documentclass{article}
\usepackage{hyperref,url,float,graphicx,amsmath,amsthm,amssymb} 
\usepackage[margin=1.753in]{geometry}

\newtheorem*{theorem*}{Theorem}
\newtheorem{corollary}{Corollary}
\newtheorem{proposition}{Proposition}
\theoremstyle{definition}

\title{Greedy fusion}
\author{Andrew Schopieray}
\date{}

\begin{document}

\maketitle

\begin{abstract}
    We demonstrate that certain sequences of integer formal codegrees, motivated by the greedy algorithm for unit fraction decompositions of $1$, dictate the structure of fusion rings.  In particular, no fusion ring that has formal codegrees $2$, $3$, and $7$ is categorifiable.
\end{abstract}

\paragraph{Fusion rings, physical applications, and formal codegrees}  The ubiquity of fusion rings in mathematics and physics is derived from their conceptual simplicity.  For example, they arise when studying the representation theory of Hopf and vertex operator algebras, and also in the study of anyon models with applications to quantum computation.  Interested readers may refer to \cite{MR3242743,MR4642115,ericwang} and the references therein for more perspective.  The relaxed nature of the definition of fusion rings also makes it difficult to find a use for almost any of them.  The fusion rings that have been most applicable are those that are Grothendieck rings of fusion categories in the sense of \cite{MR3242743}.  Demonstrating that a fusion ring is \emph{categorifiable} in this way can be, at worst, a time-consuming and laborious process. So for the sake of all relevant shareholders across mathematics and physics, it is overwhelmingly helpful to have some necessary conditions for fusion rings to be categorifiable.  Many of these necessary conditions are a consequence of representation theory, including the \emph{formal codegrees} \cite{MR2576705}: a multiset of positive real numbers corresponding to irreducible representations. The formal codegrees of categorifiable fusion rings must satisfy a deluge of number-theoretic constraints that help to winnow the important fusion rings from the chaff.  For example, formal codegrees of categorifiable fusion rings must be totally greater than or equal to $1$, cyclotomic algebraic integers, algebraic $d$-numbers,  and must divide the global dimension of their categorification.  One can refer to \cite{MR2576705,MR3427429} for proofs of these results and many more. The most widespread formal codegrees are positive integers, which trivially satisfy most all number-theoretic requirements in practice. Each positive integer, for example, appears as the formal codegree of a categorifiable fusion ring, with $1$ only being a formal codegree of the trivial fusion ring, but $2$ appearing as a formal codegree of all of the character rings of the dihedral groups of order $2n$ with $n$ an odd integer.  All representation categories of finite groups, and more generally all integral fusion categories, have integer formal codegrees. So answering the question of what combinations of integers, or more generally algebraic integers, arise as formal codegrees of fusion categories has both modern and classical importance.

\par In this note, we characterize the \emph{greedy fusion rings} whose formal codegrees follow Sylvester's sequence (see (\ref{sylvester})).  These are extreme examples of near-integral fusion rings \cite{dong2024nearintegralfusion} in the sense that they are fusion rings of rank $n$ with a fusion subring of rank $k$ for all $1\leq k\leq n$.  Our main result (Proposition \ref{thetheorem}) is that any fusion ring whose formal codegrees contain a truncation of Sylvester's sequence must contain a greedy fusion subring.  As a corollary, any fusion ring of rank $n+1$ whose formal codegrees include the first $n$ entries of Sylvester's sequence is isomorphic to a greedy fusion ring.
\par The consequences of results such as these for the study of fusion categories are strong.  For example, there are infinitely many fusion rings of rank greater than $3$ whose formal codegrees contain $2$, $3$, and $7$, and also satisfy all number-theoretic restrictions which are currently known.  But our main result implies such a fusion ring contains the greedy fusion ring of rank $4$, which is not categorifiable \cite[Lemma 5.4]{alekseyev2025classifyingintegralgrothendieckrings}.  One can also interpret these results anywhere fusion categories arise; for example, our results can be written in terms of the existence or nonexistence of certain conjugacy class sizes of finite groups (outlined in the following section), dimensions of irreducible representations of Hopf algebras, dimensions of simple objects in modular fusion categories, etc.  We end our note with open-ended questions to motivate further investigations.

\paragraph{Basic notions} Let $R$ be a ring with identity $1_R$ which is a free abelian group under the operation of addition, and let $B=\{b_1,\ldots,b_n\}$ for some $n\in\mathbb{Z}_{\geq1}$ be a distinguished basis such that if $b_jb_k=\sum_{\ell=1}^nc_{j,k}^\ell b_\ell$, then $c_{j,k}^\ell\in\mathbb{Z}_{\geq0}$.  Such a ring with $1_R\in B$ has been referred to as a unital $\mathbb{Z}_{\geq0}$-ring of finite rank in the literature \cite[Section 3.1]{MR3242743}.  Many unital $\mathbb{Z}_{\geq0}$-rings have an additional symmetry which we will refer to as duality.  This symmetry is an involution $b\mapsto b^\ast$ of $B$ such that, when extended linearly, is an anti-involution of $R$, and $1_R$ is a summand of $bb^\ast$ with multiplicity $1$ while $1_R$ does not appear in $bb'$ for any other basis elements $b'\neq b^\ast$.  A unital $\mathbb{Z}_{\geq0}$-ring with duality is a \emph{fusion ring}.

\par There are only a few basic notions and terms we will need regarding fusion rings.  First, the basis $B$ will be implied in future statements; the cardinality of $B$ is known as the rank of $R$. Second, the nonnegative integer structure coefficients $c_{j,k}^\ell$ in our definition are referred to as the fusion rules or fusion coefficients of $R$.  Thirdly, there is a unique ring homomorphism $\mathrm{FPdim}:R\to\mathbb{C}$ which takes nonnegative values on $B$ known as the Frobenius-Perron dimension \cite[Section 3.3]{MR3242743}. And lastly, the eigenvalues of multiplication by $\sum_{b\in B}bb^\ast$ in $R$ are known as the \emph{formal codegrees} of $R$.  These eigenvalues are traces of irreducible representations $\varphi$ of $R\otimes_\mathbb{Z}\mathbb{C}$ appearing with multiplicity $\dim(\varphi)$ so we will often denote them by $f_\varphi$ to indicate the corresponding representation.  For example, if $\varphi:R\to\mathbb{C}$ is any ring homomorphism, which we may briefly call a character of the ring, the corresponding formal codegree is simply $f_\varphi=\sum_{b\in B}\varphi(b)\varphi(b^\ast)=\sum_{b\in B}|\varphi(b)|^2$.  In particular, we denote the formal codegree $f_{\mathrm{FPdim}}=:\mathrm{FPdim}(R)$ and refer to it as the Frobenius-Perron dimension of $R$; this is always the largest formal codegree of $R$, but other irreducible representations may have this same formal codegree as well. We will abbreviate $d_j:=\mathrm{FPdim}(b_j)$ for any indexed basis elements of $R$ for brevity.

\newpage

\par The most familiar example of fusion rings to many are the integral group rings $\mathbb{Z}G$ for finite groups $G$.  The formal codegrees in this case correspond to irreducible representations $\varphi$ of $G$ and the corresponding formal codegree is $|G|/\dim(\varphi)$.  Similarly, the character ring $R_G$ of $G$ has formal codegrees $|G|/|C|$ where $|C|$ is the order of a conjugacy class of $G$.  For example, consider $D_5$, the dihedral group of order $10$. The integral group ring $\mathbb{Z}D_5$ has formal codegrees $10$ and $10$ with multiplicity $1$ corresponding to the two distinct one-dimensional representations of $D_5$, and $5$ and $5$, each with multiplicity $2$, corresponding to the two nonisomorphic two-dimensional irreducible representations of $D_5$. Alternatively, $R_{D_5}$ has formal codegrees $2$, $5$, $10$, and $10$ all with multiplicity $1$, corresponding to the four conjugacy classes of $D_5$.  Note that $\mathrm{FPdim}(\mathbb{Z}G)=\mathrm{FPdim}(R_G)$ although they often have different formal codegrees; in general formal codegrees are not enough to distinguish fusion rings; the smallest example is $\mathbb{Z}C_2^2$ versus $\mathbb{Z}C_4$.  The rings $\mathbb{Z}G$ and $R_G$ are integral in the sense that the Frobenius-Perron character $\mathrm{FPdim}:R\to\mathbb{C}$ is integer-valued.

\par We say a fusion ring $R$ is \emph{categorifiable} if $R$ is the Grothendieck ring of a fusion category over $\mathbb{C}$ in the sense of \cite{MR3242743}. For example the fusion rings $\mathbb{Z}G$ are categorifiable, being the Grothendieck rings of $\mathrm{Vec}_G$, the category of $G$-graded vector spaces, and the fusion rings $R_G$ are categorifiable, being the Grothenieck rings of $\mathrm{Rep}(G)$, the categories of representations of $G$. It is clear that our main result applies to integral group rings: the dimension of irreducible representations of a finite group $G$ are bounded by $\sqrt{|G|}$, hence $f_\varphi=|G|/\dim(\varphi)\geq\sqrt{|G|}$ which implies $2$ is a formal codgree only if $|G|\leq4$.  All such groups are abelian so we conclude $\mathbb{Z}C_2$ is the only example of a formal codegree of $2$ in this family.  It is less obvious that our main result applies to the character rings $R_G$. A formal codegree of $2$ corresponds to a conjugacy class $C$ of $G$ with $|G|=2|C|$, and we encourage the reader to show as an exercise that this limits $G$ to a certain family of Frobenius groups, for which conjugacy classes of size $|G|/3$ and $|G|/7$ cannot simultaneously occur.  For example, the character rings $R_{D_n}$ with $n$ an odd integer all have formal codegree $2$, but only $R_{D_3}$ has formal codegree $3$ and only $R_{D_7}$ has formal codegree $7$.

\paragraph{Formal codegrees and Sylvester's sequence} Let $R$ be a fusion ring and $\mathrm{Irr}(R)$ be the set of isomorphism classes of irreducible representations of $R\otimes_\mathbb{Z}\mathbb{C}$.  The formal codegrees of $R$ satisfy the numerical constraint \cite[Proposition 2.10]{MR3427429}:
\begin{equation}\label{constraint}
    1=\sum_{\varphi\in\mathrm{Irr}(R)}\frac{\dim(\varphi)}{f_\varphi}.
\end{equation}
For example, let $G$ be the order-24 group $G:=\mathrm{SL}_2(\mathbb{F}_3)$. The group $G$ has conjugacy classes of size $1,1,4,4,4,4,6$, so the formal codegrees of $R_G$ satisfy
\begin{equation}
    1=\frac{1}{4}+\frac{1}{6}+\frac{1}{6}+\frac{1}{6}+\frac{1}{6}+\frac{1}{24}+\frac{1}{24}.
\end{equation}
This illustrates that if the formal codegrees of a fusion ring are all integers, (\ref{constraint}) is akin to finding unit fraction decompositions of $1$.  This viewpoint was most notably leveraged to abstract algebra by E.\ Landau \cite{MR1511192} to show that for any positive integer $k$, there are finitely many finite groups up to isomorphism with exactly $k$ conjugacy classes.  We guide the interested reader to the survey paper \cite{MR4544007} for more on this historically significant subject.  Similar techniques have been fundamental to studying integral fusion categories and their relatives.  The bottleneck to this technique is that the number of decompositions of $1$ into $k$ unit fractions grows quickly, to the extent that the total number of such decompositions is only known at this time for very small $k$.  We strongly emphasize that for arbitrary fusion rings this exact approach is entirely inadequate, since even for fusion rings of rank $2$, there are infinitely many decompositions of $1$ if irrational formal codegrees are allowed. Specifically, for a fusion ring with basis $\{1_R,x\}$ and nontrivial fusion rule $xx=1_R+kx$ for $k\in\mathbb{Z}_{\geq0}$, the formal codegrees of $R$ satisfy
\begin{equation}
    1=\frac{1}{(1/2)(k^2+4+k\sqrt{k^2+4})}+\frac{1}{(1/2)(k^2+4-k\sqrt{k^2+4})}.
\end{equation}
When all of the formal codegrees of a fusion ring are integers, the numerical constraint in (\ref{constraint}) roughly says that the formal codegrees cannot all be small, i.e.\ each unit fraction in the sum cannot be very large; the strategy of choosing the largest unit fraction to be included without the sum reaching $1$ is known as the \emph{greedy algorithm} which has lead to numerous interesting number-theoretic investigations over the past centuries.

\par Related to the greedy algorithm is \emph{Sylvester's sequence} \cite[OEIS A000058]{oeis}: a positive integer sequence defined by $a_{n+1}:=1+\prod_{k=0}^na_k$ for all $n\in\mathbb{Z}_{\geq1}$ with $a_0=2$. The first few entries, all that fit on a single line, are
\begin{equation}\label{sylvester}
2, 3, 7, 43, 1807, 3263443, 10650056950807, 113423713055421844361000443.    
\end{equation}
The product definition of Sylvester's sequence is equivalent to a sum definition, which we will need for our main result: $a_0=2$ and
\begin{equation}\label{quadraticrecursion}
a_{n+1}=1+\prod_{k=0}^na_k=1+a_n\prod_{k=0}^{n-1}a_k=1+a_n(a_n-1)=a_n^2-a_n+1.   
\end{equation}
One should prove as an exercise $\sum_{k=0}^{n-1}\frac{1}{a_k}=\frac{a_n-2}{a_n-1}$ for $n>0$, which implies that Sylvester's sequence describes the least solution $x_0\leq x_1\leq\cdots\leq x_n$ to $\sum_{j=0}^n\frac{1}{x_j}$ when ordered lexicographically, which is $x_0=a_0,x_1=a_1,\ldots,x_n=a_n-1$.

\par For $n\in\{1,2\}$, the solutions $1=1/2+1/2$ and $1=1/2+1/3+1/6$ are realized by the formal codegrees of fusion rings ($\mathbb{Z}C_2$ and $R_{D_3}$) and moreover by categorifiable fusion rings.  Although there are fusion rings whose formal codegrees realize these least solutions generated by Sylvester's sequence for all positive integers $n$, we will show they are only categorifiable in these two cases.

\paragraph{Greedy fusion and the main result} Define the \emph{greedy fusion rings} $R_n$ for $n\in\mathbb{Z}_{\geq1}$ inductively starting with $R_1\cong\mathbb{Z}C_2$, and $R_n$ for $n\in\mathbb{Z}_{\geq2}$ as a fusion ring of rank $n+1$ containing $R_{n-1}$ as a fusion subring with basis $\{r_0,\ldots,r_{n-1}\}$ ordered by increasing Frobenius-Perron dimension, and one additional basis element $r_n$.  Since $R_{n-1}$ is a fusion subring, it is closed under duality, hence $r_k^\ast=r_k$ inductively for all $0\leq k\leq n$.

\par For all $0\leq k\leq n-1$, $c_{n,k}^k=c_{k,n}^k=c_{k,k}^n=0$ \cite[Proposition 3.1.6]{MR3242743} and by dimension counting, $c_{n,n}^k=c_{n,k}^{n}=c_{k,n}^{n}=d_k$. In particular, there is only one unknown fusion rule, which we choose to be $c_{n,n}^{n}:=a_{n-1}-2$. In other words,
\begin{equation}\label{fusion}
 r_nr_n=(a_{n-1}-2)r_n+\sum_{j=0}^{n-1}d_jr_j.
\end{equation}
Any choice of fusion coefficient $c_{n,n}^n\in\mathbb{Z}_{\geq0}$ produces a fusion ring with formal codegrees $a_0,a_1,\ldots,a_{n-2}$; our particular choice is only to ensure the remaining two formal codegrees are related to Sylvester's sequence as well.   So with this choice, we have an associative (and commutative) product, and hence a fusion ring of rank $n+1$ and we claim that, by induction, $d_n=a_{n-1}-1$ and $\mathrm{FPdim}(R_n)=a_n-1$. This is true for $R_1$, and measuring the Frobenius-Perron dimension of (\ref{fusion}) yields
\begin{align}
d_n^2=(a_{n-1}-2)d_n+\sum_{j=0}^{n-1}d_j^2&=(a_{n-1}-2)d_n+\mathrm{FPdim}(R_{n-1}) \\
&=(a_{n-1}-2)d_n+(a_{n-1}-1)
\end{align}
by the inductive hypothesis. So we have
\begin{equation}
d_n=\frac{1}{2}\left(a_{n-1}-2+\sqrt{(a_{n-1}-2)^2+4(a_{n-1}-1)}\right)=a_{n-1}-1
\end{equation}
by the recursive relation in (\ref{quadraticrecursion}). By this same relation, we verify that
\begin{align}
    \mathrm{FPdim}(R_n)=d_n^2+\sum_{k=0}^{n-1}d_k^2&=(a_{n-1}-1)^2+(a_{n-1}-1)=a_n-1.
\end{align}
By construction, we have a chain of containment $R_1\subset R_2\subset\cdots\subset R_n$ for all $n\in\mathbb{Z}_{\geq1}$.  The character table of $R_n$ is illustrated in Figure \ref{fig:char} and verified in the proof of Proposition \ref{thetheorem} with $f_{\psi_j}=a_{j-1}$ for all $1\leq j\leq n$.

\begin{figure}[H]
    \centering
    \begin{equation}
    \begin{array}{|c|c|cccccc|}\hline
         & \mathrm{FPdim} & \psi_1 & \psi_2 & \psi_3 & \psi_4 & \cdots & \psi_n \\\hline
         r_0 & 1 & 1 & 1 & 1 & 1 & \cdots & 1 \\
         r_1 & 1 & -1 & 1 & 1 & 1 & \cdots & 1 \\
         r_2 & 2 & 0 & -1 & 2 & 2 &\cdots & 2 \\
         r_3 & 6 & 0 & 0 & -1 & 6 & \cdots & 6 \\
         r_4 & 42 & 0 & 0 & 0 & -1 & \cdots & 42 \\
         \vdots & \vdots & \vdots & \vdots& \vdots & \vdots & \ddots & \vdots \\
         r_n & a_{n-1}-1 & 0 & 0 & 0 & 0 & \cdots & -1 \\\hline
         \end{array}
\end{equation}
    \caption{The character table of the greedy fusion ring $R_n$}
    \label{fig:char}
\end{figure}

\begin{proposition}\label{thetheorem}
    Let $R$ be a fusion ring and $n\in\mathbb{Z}_{\geq1}$.  If the formal codegrees of $R$ include $a_0,a_1,\ldots,a_{n-1}$ then $R_n\subset R$ is a fusion subring.
\end{proposition}

\begin{proof}
    Assume for some $n\in\mathbb{Z}_{\geq1}$, that $R$ is a fusion ring with formal codegrees $a_0,a_1,\ldots,a_{n-1}$.  Note that all of these formal codegrees appear with multiplicity $1$ since they are chosen greedily, i.e.\ for all $0\leq j<n$, $a_j$ has multiplicity $1$ since
    \begin{equation}
        \frac{1}{a_j}+\sum_{i=0}^j\frac{1}{a_i}>1.
    \end{equation}
    Therefore, each formal codegree $a_j$ corresponds to a distinct ring homomorphism $\psi_{j+1}:R\to\mathbb{C}$ such that $f_{\psi_{j+1}}=a_j$ for all $0\leq j\leq n-1$, hence the $\psi_{j+1}$ are all integer-valued. We will show inductively that $R_1\subset R_2\subset\cdots R_n\subset R$ and that the character table of $R_n$ coincides with Figure \ref{fig:char}.
    \par For the base case, a formal codegree of $a_0=2$ corresponding to an integer-valued character $\psi_1:R\to\mathbb{C}$ implies $\psi_1(r_1)=\pm1$ for a unique nontrivial basis element $r_1\in R$ and $\psi_1(b)=0$ for all other nontrivial basis elements $b$. Orthogonality of $\psi_1$ with $\mathrm{FPdim}$ \cite[Lemma 8.14.1]{MR3242743} implies $d_1=-1/\psi_1(r_1)$ thus $\psi_1(r_1)=-1$, and moreover $d_1=1$.  Since $\psi_1(b)=0$ for all other nontrivial basis elements $b$ of $R$, then $ R_\mathrm{pt}\cong\mathbb{Z}C_2\cong R_1$.
    
    \par Now assume that, for some $k\in\mathbb{Z}_{\geq0}$ with $k\leq n$, we have shown there exists a fusion subring $R_{k-1}\subset R$ with basis $r_0,r_1,\ldots,r_{k-1}$ such that $\psi_{k-1}(r_j)=d_j$ for all $0\leq j<k-1$, $\psi_{k-1}(r_{k-1})=-1$, and $\psi_{k-1}(b)=0$ for all other basis elements $b$ of $R$. Consider the integer-valued character $\psi_k$ with $f_{\psi_k}=a_{k-1}$.  Orthogonality with $\psi_1,\ldots,\psi_{k-1}$ implies, iteratively, that $\psi_k(r_j)=d_j$ for all $0\leq j<k$, hence
    \begin{equation}
        \sum_{j=0}^{k-1}\psi_k(r_j)^2=\sum_{j=0}^{k-1}d_j^2=\mathrm{FPdim}(R_{k-1})=a_{k-1}-1.
    \end{equation}
    Therefore $\psi_k$ being integer-valued implies there exists a unique basis element $r_k$ with $r_k\neq r_j$ for $0\leq j<k$ such that $\psi_k(r_k)=\pm1$ and as before, orthogonality of $\psi_k$ with $\mathrm{FPdim}$ implies $0=\sum_{j=0}^{k-1}d_j^2\pm d_k$ and therefore $\psi_k(r_k)=-1$ and moreover $d_k=a_{k-1}-1$.
    
    \par It remains to show that $r_0,\ldots,r_k$ generate a fusion subring isomorphic to $R_k$.  To this end, since $R_{k-1}\subset R$ is a fusion subring, $0=c_{j,j}^k=c_{j,k}^j=c_{k,j}^j$    for all $0\leq j<k$ \cite[Proposition 3.1.6]{MR3242743}.  Now applying the character $\psi_k$ to the fusion of $r_j$ and $r_k$ for $0\leq j<k$ yields $-d_j=-c_{j,k}^k$, hence $c_{k,k}^j=c_{j,k}^k=d_j$.  Applying the character $\psi_k$ to the fusion of $r_k$ with itself gives $1=\sum_{j=0}^{k-1}d_j^2-c_{k,k}^k$, hence $c_{k,k}^k=a_{k-1}-2$.  Lastly, we compute
    \begin{equation}
        \sum_{j=0}^{k-1}d_j^2+(a_{k-1}-2)d_k=(a_{k-1}-1)+(a_{k-1}-2)(a_{k-1}-1)=(a_{k-1}-1)^2=d_k^2.
    \end{equation}
    Therefore $c_{k,k}^b=0$ for all basis elements different from $r_0,\ldots,r_k$.  Moreover, we have shown $r_0,\ldots,r_k$ form a basis for a fusion subring $R_k\subset R$.
\end{proof}

\begin{corollary}
    Let $R$ be a fusion ring of rank $n+1$ for some $n\in\mathbb{Z}_{\geq1}$.  Then the formal codegrees of $R$ include $a_0,a_1,\ldots,a_{n-1}$ if and only if $R\cong R_n$.
\end{corollary}

\begin{corollary}
    No fusion ring with formal codegrees $2$, $3$, and $7$ is categorifiable.
\end{corollary}

\begin{proof}
    Proposition \ref{thetheorem} implies that such a fusion ring contains a fusion subring isomorphic to $R_3$, which is not categorifiable. A terse proof is given in \cite[Lemma 5.4]{alekseyev2025classifyingintegralgrothendieckrings} via computer and leveraging previous results on premodular categories. To fill in the details, one can compute by hand that the double of any categorification of $R_3$ has a Tannakian subcategory with the fusion rules of $R_3$, but there is no group of order $42$ with irreducible representations of dimensions $1$, $1$, $2$, and $6$.
\end{proof}

\paragraph{Consequences and future directions} Formal codegrees and unit fraction decompositions of $1$ have been utilized in numerous efforts to classify fusion categories of small rank (refer to \cite{alekseyev2025classifyingintegralgrothendieckrings,alekseyev2026classificationintegralmodulardata,MR2869100}, for example and references therein).  In particular, \cite[Section 2.4]{alekseyev2025classifyingintegralgrothendieckrings} contains more questions related to this paper.  The combination of $1/2$, $1/3$, and $1/7$ does not often appear in unit decompositions of $1$ at small ranks, but for ranks larger than $3$ it always appears, and appears with more and more frequency as the rank increases.  For example, there are over $20\,000$ unit fraction decompositions of $1$ with $7$ terms that include $1/2$, $1/3$, and $1/7$.  Identifying sequences of integers that can never occur in the formal codegrees of a categorification therefore decreases the computational effort required for these tasks. This also begs the question of whether other combinations of positive integers are forbidden from categorifications, or even from fusion rings themselves. Clearly the constraint in (\ref{constraint}) should be incorporated into the question to make it nontrivial, so we ask
\begin{center}
    \emph{Which sets of positive integers $\{x_1,\ldots,x_n\}\subset\mathbb{Z}_{\geq2}$ such that $\sum_{j=1}^n\frac{1}{x_j}<1$ never simultaneously appear as the formal codegrees of a fusion ring/category?}
\end{center}
For example, the reader should convince themself as an exercise that there does not exist a fusion ring whose formal codegrees include $2$, $4$, and $5$. Combinations that cannot produce sensible fusion at all are, of course, less interesting than those that are prevalent among fusion rings but nonexistent among fusion categories.

\par Lastly, note that the existence of the formal codegree $2$ implies the existence of a nontrivial invertible element $x$ of order $2$ by Proposition \ref{thetheorem}.  If such a fusion ring is categorified by a spherical fusion category (see \cite[Section 4.7]{MR3242743}), then the second Frobenius-Schur indicator $\nu_2$ of $x$ is $\pm1$.  It was shown by the author in \cite[Corollary 3.13]{MR4655273} that $\nu_2(x)=-1$ implies that $x$ acts fixed-point freely on all basis elements.  But if $x$ of order $2$ arises from the existence of a formal codegree $f_\psi=2$ for a ring homomorphism $\psi:R\to\mathbb{C}$, then unless $\psi=\mathrm{FPdim}$, i.e. $R\cong\mathbb{Z}C_2$, we may apply $\psi$ to the fusion $yy^\ast$ for any other nontrivial basis element $y$ to yield $0=|\psi(y)|^2=1-c_{y,y^\ast}^x$.  Therefore $c_{y,y^\ast}^x=c_{x,y}^y=1$. Hence $x$ fixes all nontrivial basis elements except itself, and moreover $\nu_2(x)=1$.  It is likely that other categorical properties like Frobenius-Schur indicators can be deduced solely from the existence of small integer formal codegrees.  More research in this direction may illustrate more formal codegrees that are incompatible with categorification, or prove structural results about categorifications, when they exist.



\bibliographystyle{plainurl}
\bibliography{bib}

\end{document}